\documentclass[12pt]{amsart}
\pdfoutput=1
\usepackage[margin=1in]{geometry}
\usepackage[utf8]{inputenc}
\usepackage{graphicx}
\usepackage{amsmath,amssymb}
\usepackage{amsthm}
\usepackage{color}
\usepackage{tikz}
\usepackage[utf8]{inputenc}
\usepackage{pgfplots}
\pgfplotsset{compat=1.18}
\usetikzlibrary{fillbetween}
\allowdisplaybreaks

\makeatletter
\def\l@section{\@tocline{1}{12pt plus2pt}{0pt}{}{\bfseries}}
\def\l@subsection{\@tocline{2}{0pt}{2pc}{2pc}{}}
\makeatother
\makeatletter
\def\subsection{\@startsection{subsection}{2}{\z@}%
	{-3.25ex\@plus -1ex \@minus -.2ex}%
	{1.5ex \@plus .2ex}%
	{\normalfont\bfseries\boldmath}}
\def\subsubsection{\@startsection{subsubsection}{3}%
	\z@{.5\linespacing\@plus.7\linespacing}{-.5em}%
	{\normalfont\bfseries\boldmath}}
\renewcommand\paragraph{\@startsection{paragraph}{4}{\z@}%
	{3.25ex \@plus1ex \@minus.2ex}%
	{-1em}%
	{\normalfont\normalsize\bfseries}}
\makeatother

\theoremstyle{plain}
\newtheorem{thm}{Theorem}[section]

\newtheorem{cor}[thm]{Corollary}
\newtheorem{lem}[thm]{Lemma}
\newtheorem{prop}[thm]{Proposition}
\newtheorem{alg}[thm]{Algorithm}
\theoremstyle{definition}

\theoremstyle{remark}

\theoremstyle{plain}

\numberwithin{equation}{section}

\newtheorem{opques}[thm]{Open Question}

\theoremstyle{plain} 
\newcommand{\thistheoremname}{}
\newtheorem{genericthm}[thm]{\thistheoremname}

  \newtheorem*{genericthm*}{\thistheoremname}
\newenvironment{namedthm*}[1]
  {\renewcommand{\thistheoremname}{#1}%
   \begin{genericthm*}}
  {\end{genericthm*}}

\newcommand{\D}{{\mathbb D}}

\newcommand{\R}{{\mathbb R}}

\newcommand{\C}{{\mathbb C}}

\newcommand{\Z}{{\mathbb Z}}
\newcommand{\calR}{{\mathcal R}}

\newcommand{\calB}{{\mathcal B}}

\newcommand{\supp}{{\textnormal{supp}}}

\makeatletter
\newcommand{\vast}{\bBigg@{4}}
\newcommand{\Vast}{\bBigg@{5}}
\makeatother

\def\udot#1{\ifmmode\oalign{$#1$\crcr\hidewidth.\hidewidth
    }\else\oalign{#1\crcr\hidewidth.\hidewidth}\fi}

\def\R{\mathbb{R}}
\def\Z{\mathbb{Z}}
\def\T{\mathbb{T}}
\def\C{\mathbb{C}}

\def\beq{\begin{equation}}
\def\eeq{\end{equation}}
\makeatletter
\newcommand{\doublewidetilde}[1]{{%
  \mathpalette\double@widetilde{#1}%
}}
\newcommand{\double@widetilde}[2]{%
  \sbox\z@{$\m@th#1\widetilde{#2}$}%
  \ht\z@=.9\ht\z@
  \widetilde{\box\z@}%
}
\makeatother

\usepackage{tikz}

\makeatletter
\def\@makefnmark{%
  \leavevmode
  \raise.9ex\hbox{\fontsize\sf@size\z@\normalfont\tiny\@thefnmark}}
\makeatother

\begin{document}
	
\title[]{On the growth of Bloch functions}

\author{Bingyang Hu}
\address{(Bingyang Hu) Department of Mathematics and Statistics\\
        Auburn University\\
        Auburn, Alabama, U.S.A, 36849}
\email{bzh0108@auburn.edu}

\author{Jie Xiao}
\address{(Jie Xiao) Department of Mathematics and Statistics\\
        Memorial University of Newfoundland\\
        St. John's, NL, Canada, A1C 5S7}
\email{jxiao@mun.ca}

\author{Xiaojing Zhou}
\address{(Xiaojing Zhou) Department of Mathematics and Statistics\\
         Auburn University\\
         Auburn, Alabama, U.S.A, 36849}
\email{xiz0003@auburn.edu}

\begin{abstract}
We prove that there exist two Bloch functions $f_1$ and $f_2$ on $\D$ such that
$$
|f_1(z)|+|f_2(z)|
\geq
\left(\log\frac{1}{1-|z|}\right)^{1/2},
\qquad z\in\D,
$$
thereby resolving an open problem posed in 2008 by Girela, Pel\'aez, P\'erez-Gonz\'alez and R\"atty\"a. Our proof is based on a new Szeg\H{o}-type recursion involving $\C^2$-valued polynomials and their reciprocal polynomials.
\end{abstract}

\date{\today}


\maketitle

\section{Introduction}

Let $\D$ denote the unit disc in $\C$, let $\T:=\partial\D$ be the unit circle, and let $H(\D)$ denote the space of all holomorphic functions on $\D$, equipped with the compact--open topology.

The main \emph{goal} of this paper is to resolve the following open problem in the theory of Bloch spaces, posed by Girela, Pel\'aez, P\'erez-Gonz\'alez and R\"atty\"a \cite{GirelaPelaezPerezGonzalezRattya2008} in 2008.

\begin{opques} \label{20260724op01}
Do there exist two functions $f_1, f_2 \in \calB$ such that
\begin{equation} \label{20260724eq01}
|f_1(z)|+|f_2(z)| \ge \left(\log \frac{1}{1-|z|} \right)^{1/2}, \qquad z \in \D. 
\end{equation} 
\end{opques}

We first recall the following very nice observation due to Girela, Pel\'aez, P\'erez-Gonz\'alez and R\"atty\"a \cite{GirelaPelaezPerezGonzalezRattya2008}. The factor
$\left(\log \frac{1}{1-|z|}\right)^{1/2}$ in \eqref{20260724eq01} is sharp in the following sense: if $f_1, f_2 \in \calB$ and $\Phi:[0,1)\to[0,\infty)$ satisfies
\begin{equation}\label{20260724eq02}
    |f_1(z)|+|f_2(z)|\geq \Phi(|z|),
    \qquad z\in\D,
\end{equation}
then
$$
\Phi(r)\lesssim
\left(\log\frac{e}{1-r}\right)^{1/2},
\qquad 0\leq r<1.
$$
Indeed, taking the contour integral on both sides of \eqref{20260724eq02} and applying the Cauchy--Schwarz inequality, we obtain, for every $0<r<1$,
$$
\begin{aligned}
\Phi(r)
&\leq
\sum_{\ell=1}^2
\frac{1}{2\pi}
\int_0^{2\pi}
|f_\ell(re^{it})|\,dt \\
&\leq
\sum_{\ell=1}^2
\left(
\frac{1}{2\pi}
\int_0^{2\pi}
|f_\ell(re^{it})|^2\,dt
\right)^{1/2}
\lesssim
\left(\log\frac{e}{1-r}\right)^{1/2}.
\end{aligned}
$$
This proves the assertion. In the last estimate, we used
\cite[(1.4)]{GirelaPelaezPerezGonzalezRattya2008}; see also \cite{ClunieMacGregor1984,Makarov1985}.

\vspace{0.1cm}

It therefore remains to determine whether there exists a pair of functions $f_1$ and $f_2$ for which the threshold
$\left(\log \frac{1}{1-|z|}\right)^{1/2}$ is attained. Our main theorem shows that the answer to Open Question~\ref{20260724op01} is \emph{affirmative}.

\begin{thm} \label{Blochmainthm}
There exist two functions $f_1, f_2 \in \calB$ such that
\begin{equation} \label{20260725eq20}
|f_1(z)|+|f_2(z)| \ge \left(\log \frac{1}{1-|z|} \right)^{1/2}, \qquad z \in \D. 
\end{equation} 
\end{thm}
Here, we recall that the Bloch space $\calB$ consists of all holomorphic functions $f \in H(\D)$ such that
$$
\|f\|_{\calB, *}:=\sup_{z \in \D} (1-|z|^2)|f'(z)|<+\infty. 
$$
It is well known that $\|\cdot\|_{\calB,*}$ defines a semi--norm on $\calB$, and that $\calB$ is a Banach space under the norm $\|f\|_{\calB}:=|f(0)|+\|f\|_{\calB,*}$. 

\vspace{0.1cm}

The estimate \eqref{20260725eq20} in Theorem~\ref{Blochmainthm} belongs to the general class of \emph{growth problems} in complex function theory, which ask whether finitely many functions in a given holomorphic function space can jointly attain its maximal pointwise growth. Such problems are well understood for the Bergman-type growth spaces
$$
A^{\alpha}
:=
\left\{
f\in H(\D):
\sup_{z\in\D}(1-|z|^2)^\alpha |f(z)|<\infty
\right\},
\qquad \alpha>0,
$$
and their weighted generalizations; see Ramey and Ullrich
\cite{RameyUllrich1991} and Abakumov and Doubtsov
\cite{AbakumovDoubtsov2012,AbakumovDoubtsov2015}. This framework also includes the earlier result of the second author for the weighted Bloch spaces $\calB_\alpha$ when $\alpha>1$; see \cite[Lemma~3.1]{Xiao2004}. Indeed, in this range, $\calB_\alpha$ coincides with $A^{\alpha-1}$, with equivalent norms.

The constructions in these works are based on lacunary series, including weakly lacunary series in the general weighted setting. This method, however, breaks down for Bloch space. Indeed, a Hadamard gap series belongs to $\calB$ precisely when its coefficients are bounded (see, e.g., \cite{Yamashita1980}), whereas the classical lacunary construction for an increasing radial weight requires unbounded coefficients. Consequently, it cannot produce the square-root logarithmic growth required in Theorem~\ref{Blochmainthm}.

\vspace{0.1cm}

Our proof instead uses a new \emph{Szeg\H{o}-type recursive construction} for $\C^2$-valued polynomials and their reciprocal polynomials, inspired by the Szeg\H{o} recursion from the theory of orthogonal polynomials on the unit circle (OPUC); see Simon's expository paper \cite{Simon2005}. 

The rest of the paper is organized as follows. In Section~\ref{Sec02}, we develop a recursive construction of $\C^2$-valued reciprocal polynomials and show that its limiting coordinates define two Bloch functions. In Section~\ref{Sec03}, we prove that these functions satisfy the desired lower bound, first near $\T$ and then throughout $\D$, thereby completing the proof of Theorem~\ref{Blochmainthm}.

Finally, throughout the paper, for nonnegative quantities $a$ and $b$, we write
$a\lesssim b$ if $a\leq Cb$ for some constant $C>0$ independent of $a$ and $b$. We write
$a\simeq b$ if both $a\lesssim b$ and $b\lesssim a$ hold.
\\
\noindent{\bf Acknowledgement.}
The first author was supported by the NSF grant DMS-2555999 and by the Simons Travel grant MPS-TSM-00007213. The second author was supported by NSERC of Canada $\#$ 202979.

\bigskip 

\section{Szeg\H{o}-type recursion of vector-valued reciprocal polynomials} \label{Sec02}

\subsection{Vector-valued conjugate reciprocal polynomials}

For $u=\left(u_1,u_2\right),v=\left(v_1,v_2\right)\in\C^2$, recall that the Hermitian inner product is given by
$$
\langle u,v\rangle_{\C^2} :=u_1\overline{v_1}+u_2\overline{v_2}.
$$
The norm induced by this inner product is denoted by
$$
\|u\|_{\C^2}:=
\left(|u_1|^2+|u_2|^2\right)^{1/2}.
$$
Consider $\C^2$-valued polynomials of the form
$$
{\bf V}(z)=\left(V_1(z),V_2(z)\right)
=\sum_{k=0}^d {\bf v}_k z^k,
\qquad {\bf v}_k\in\C^2,
$$
for $d \in \Z_{\ge 0}$. Define the \emph{degree} and the
\emph{Fourier support} of ${\bf V}$, respectively, by
$$
\deg{\bf V}:=\max\left\{\deg V_1,\deg V_2\right\}
$$
and
$$
\supp\widehat{\left({\bf V}|_\T\right)}:= \left\{k \ge 0:{\bf v}_k\neq(0,0)\right\}.
$$
Here, the Fourier transform is understood as the Fourier transform from
$\T$ to $\Z$, namely,
$$
\widehat{\left({\bf V}|_\T\right)}(k)
:=
\frac{1}{2\pi}
\int_0^{2\pi}{\bf V}\left(e^{it}\right)e^{-ikt}\,dt
\in\C^2,
\qquad k\in\Z.
$$
Next, we recall the concept of reciprocal polynomials. Let $d \in \Z_{\ge 0}$ and $P(z)=\sum_{j=0}^d a_j z^j, \; a_j \in \C$ be a scalar--valued polynomial with $\deg P \le d$. Then define its associated \emph{reciprocal polynomial} $P^{\dagger, d}$ by
$$
P^{\dagger, d}(z):=z^d \overline{P \left( \overline{z}^{-1} \right)}=\overline{a_d}+\overline{a_{d-1}}z+ \dots +\overline{a_0}z^d.
$$
Moreover, for ${\bf F}=(P, Q)$ being a $\C^2$-valued polynomial with $\deg {\bf F} \le d$, define 
$$
\calR_d {\bf F}:=\left(-Q^{\dagger, d}, \; P^{\dagger, d} \right). 
$$
The following properties are immediate. For every $\zeta\in\T$, one has:
\begin{enumerate}
\item $P^{\dagger,d}(\zeta)=\zeta^d\overline{P(\zeta)}$, 
and hence
$$
\calR_d{\bf F}(\zeta)=\zeta^d\left(-\overline{Q(\zeta)},\overline{P(\zeta)}\right).
$$

\item
\begin{equation} \label{20250724eq51}
\left\|\calR_d{\bf F}(\zeta)\right\|_{\C^2}
=
\left\|{\bf F}(\zeta)\right\|_{\C^2},
\qquad\text{and}\qquad
\left\langle{\bf F}(\zeta), \; \calR_d{\bf F}(\zeta)\right\rangle_{\C^2}=0.
\end{equation} 
\end{enumerate}

\subsection{Szeg\H{o}-type recursion and Bloch estimates}

We now begin the proof of Theorem~\ref{Blochmainthm} by constructing the required pair of Bloch functions. 

\begin{alg} \label{20260724alg01}
\begin{enumerate}
\item [$\bullet$] When $n=0$, start with
$$
{\bf F}_0=(P_0, Q_0):=(1, 0). 
$$
\item [$\bullet$] Once ${\bf F}_n=(P_n, Q_n)$ is known, define 
$$
{\bf G}_{n+1}(z):=\frac{z^{2^n}}{\sqrt{n+1}} \calR_{2^n-1} {\bf F}_n(z)
$$
and 
$$
{\bf F}_{n+1}(z):={\bf F}_n+{\bf G}_{n+1}(z).  
$$
\end{enumerate} 
\end{alg}

Here, we emphasize that ${\bf G}_n$ is defined only for $n\geq 1$. The following proposition summarizes some important properties of the sequences $\{{\bf F}_n\}_{n\geq 0}$ and $\{{\bf G}_n\}_{n\geq 1}$.

\begin{prop} \label{20260724prop01}
For every $n\ge0$:
\begin{enumerate}
\item [(i)] $\deg {\bf F}_n\le 2^n-1$;
\item [(ii)] if $n\ge1$, then $\supp \; \widehat{{\bf G}_n}\subseteq \left\{2^{n-1}, \dots, 2^n-1 \right\}$;
\item [(iii)] for every $\zeta\in\T$,
$$
        \left\| {\bf F}_n(\zeta)\right\|_{\C^2}^2=n+1;
$$
\item [(iv)] for every $\zeta\in\T$,
$$
       \left\| {\bf G}_n(\zeta)\right\|_{\C^2}^2=1,
        \qquad
        \langle {\bf F}_n(\zeta), {\bf G}_{n+1}(\zeta)\rangle_{\C^2}=0.
$$
\end{enumerate}
\end{prop}

\begin{proof}
We prove by induction. Note that when $n=0$, assertions (i) and (iii) are clear. Moreover, a direct computation yields ${\bf G}_1(z)=(0, z)$, and hence for any $\zeta \in \T$, 
$$
\left\langle {\bf F}_0(\zeta), {\bf G}_1(\zeta) \right\rangle_{\C^2}= \left\langle (1, 0), (0, \xi) \right \rangle_{\C^2}=0, 
$$
which gives the second assertion in (iv). We next consider the case $n=1$. Assertions (i) and (iii) are immediate from ${\bf F}_1(z)={\bf F}_0(z)+{\bf G}_1(z)=\left(1,z\right)$. Assertion (ii) and the first assertion in (iv) follow from the identity ${\bf G}_1(z)=\left(0,1\right)z$. Finally, to verify the second assertion in (iv), we first compute
$$
{\bf G}_2(z)=\left(-\frac{z^2}{\sqrt{2}},\frac{z^3}{\sqrt{2}}\right),
$$
and hence, for every $\zeta\in\T$,
$$
\left\langle{\bf F}_1(\zeta),{\bf G}_2(\zeta)\right\rangle_{\C^2}
=\left\langle\left(1,\zeta\right),\left(-\frac{\zeta^2}{\sqrt{2}},\frac{\zeta^3}{\sqrt{2}}\right)\right\rangle_{\C^2}
=-\frac{\overline{\zeta}^2}{\sqrt{2}}+\frac{\zeta\overline{\zeta}^3}{\sqrt{2}}=0.
$$
Now we assume the assertions (i)--(iv) hold for level $n$, and we have to prove (i)--(iv) for level $n+1$. First, we observe that since $\deg {\bf F}_n \le 2^n-1$, $\deg \calR_{2^n-1} {\bf F}_n \le 2^n-1$, and hence 
\begin{equation} \label{20260724eq50}
\deg {\bf F}_n<2^n \le \deg {\bf G}_n \le 2^{n+1}-1, 
\end{equation} 
which implies $\deg {\bf F}_{n+1}=\deg {\bf G}_{n+1} \le 2^{n+1}-1$. Hence, the assertion (i) is proved. Next, by the definition of ${\bf G}_{n+1}$, every $\C^2$-valued monomial appearing in ${\bf G}_{n+1}$ has degree at least $2^n$. This, together with \eqref{20260724eq50}, proves assertion (ii). Third, by the first equation in \eqref{20250724eq51}, we have for any $\xi \in \T$, 
\begin{equation} \label{20260724eq53}
\left\|{\bf G}_{n+1}(\xi) \right\|_{\C^2}^2=\frac{1}{n+1} \cdot \left\|\calR_{2^n-1} {\bf F}_n(\xi) \right\|^2_{\C^2}=\frac{1}{n+1} \cdot \left\|{\bf F}_n(\xi) \right\|_{\C^2}^2=1, 
\end{equation} 
where in the last estimate above, we used the induction hypothesis. This proves the first equation in assertion (iv). Moreover, by the second equation in \eqref{20250724eq51}, 
\begin{equation} \label{20260724eq54}
\left \langle {\bf F}_n(\xi), {\bf G}_{n+1}(\xi) \right\rangle_{\C^2}=\frac{\left(\overline{\xi}\right)^{2^n}}{\sqrt{n+1}} \cdot \left \langle {\bf F}_n(\xi), \calR_{2^n-1} {\bf F}_{n}(\xi) \right\rangle_{\C^2}=0, 
\end{equation} 
which gives the second assertion in (iv). Finally, by \eqref{20260724eq53} and \eqref{20260724eq54}, 
$$
\left\|{\bf F}_{n+1}(\xi) \right\|_{\C^2}^2=\left\|{\bf F}_{n}(\xi) \right\|_{\C^2}^2+\left\|{\bf G}_{n+1}(\xi) \right\|_{\C^2}^2=n+2, 
$$
which finishes the proof of assertion (iii). 
\end{proof}

As a consequence of Proposition \ref{20260724prop01}, we have the following. 

\begin{prop} \label{20260724prop02}
The series
$$
{\bf F}(z):={\bf F}_0(z)+\sum_{n=1}^{\infty} {\bf G}_n(z)=:(P(z),Q(z))
$$
converges uniformly on compact subsets of $\D$.  Hence $P,Q\in H(\D)$, and ${\bf F}_N\to {\bf F}$ uniformly on compact subsets of $\D$, where
$$
{\bf F}_N(z):={\bf F}_0(z)+\sum_{n=1}^N {\bf G}_n(z).
$$
\end{prop}

\begin{proof}
For each $n \ge 1$, write
$$
{\bf G}_n(z)=z^{2^{n-1}} {\bf H}_n(z), \qquad \textrm{where} \qquad {\bf H}_n(z):=\frac{1}{\sqrt{n}} \calR_{2^{n-1}-1} {\bf F}_{n-1}(z). 
$$
We claim that
\begin{equation} \label{20260724eq60}
\left\|{\bf H}_n(z) \right\|_{\C^2} \le 1, \qquad z \in \D. 
\end{equation} 
Indeed, first note that by Proposition \ref{20260724prop01},
$$
\left\|{\bf H}_n(\zeta) \right\|_{\C^2}=\left\|{\bf G}_n(\zeta) \right\|_{\C^2}=1, \qquad \zeta \in \T. 
$$
This means that for any unit vector ${\bf u} \in \C^2$, the polynomial $z \mapsto \langle {\bf H}_n(z), {\bf u} \rangle_{\C^2}$ has boundary modulus on $\T$ at most one. Then by the maximal modulus principle
$$
\left| \langle {\bf H}_n(z), {\bf u} \rangle_{\C^2} \right| \le 1, \qquad z \in \D. 
$$
Taking the supremum in ${\bf u} \in \C^2$ with $\|{\bf u}\|_{\C^2} \le 1$ on both sides of the above estimate yields the desired claim \eqref{20260724eq60}. 

\vspace{0.1cm}

As a consequence of \eqref{20260724eq60}, we have for any $0<R<1$ and $|z|<R$, 

$$
        \sum_{n=1}^{\infty}\|{\bf G}_n(z)\|_{\C^2} \le \sum_{n=1}^\infty |z^{2^{n-1}}| \cdot  \left\|{\bf H}_n(z) \right\|_{\C^2}
        \le
        \sum_{n=1}^{\infty}R^{2^{n-1}}
        <\infty.
$$
Finally, the desired claim follows from the Weierstrass test. 
\end{proof}

Our next goal is to show that the two holomorphic functions $P$ and $Q$ constructed in Proposition \ref{20260724prop02} belong to $\calB$. For this purpose, we need the following version of the Bernstein's theorem. 

\begin{lem} \label{20260724lem01}
If ${\bf V}$ is a $\C^2$-valued polynomial of degree at most $d\ge1$, then
\begin{equation}\label{20260724eq70}
\sup_{z\in\overline\D} \|{\bf V}'(z)\|_{\C^2} \le
        \textnormal{e} \cdot d\sup_{\zeta\in\T} \|{\bf V}(\zeta)\|_{\C^2}.
\end{equation}
\end{lem}

\begin{proof}
First, let $P$ be a scalar-valued polynomial of degree at most $d$ and put
$M:=\sup_{\T}|P|$.  The $w^d P(1/w)$ extends across $w=0$ as a polynomial of degree at most $d$.  Applying the maximum modulus principle to
this polynomial gives
\begin{equation} \label{20260724eq71}
        |P(\xi)|\le|\xi|^d M,
        \qquad |\xi|\ge1.
\end{equation}
For any $z\in\overline\D$, applying Cauchy's differentiation formula on the circle
$|\xi-z|=1/d$  with \eqref{20260724eq71} yields
\begin{align} \label{20260724eq72}
        |P'(z)| &\le \frac{1}{2\pi} \int_{|\xi-z|=\frac{1}{d}} \frac{|P(\xi)|}{|\xi-z|^2} d\xi  \le \frac{1}{2\pi} \int_{|\xi-z|=\frac{1}{d}} \frac{M \max\{1, |\xi|^d\}}{|\xi-z|^2} d\xi  \nonumber \\
        & \le 
        d \left(1+\frac{1}{d} \right)^d M
        \le \textrm{e} \cdot dM, 
\end{align} 
where we have the estimate that $|\xi| \le 1+1/d$ for $\xi$ belonging to the circle $|\xi-z|=1/d$.

Now fix a unit vector ${\bf u}\in\C^2$ and apply the estimate \eqref{20260724eq72} to
$P_{\bf u}(z):=\langle {\bf V}(z), {\bf u} \rangle_{\C^2}$.  Its boundary modulus on $\T$ is at most
$\sup_{\T}\left\|{\bf V}\right\|_{\C^2}$.  Taking the supremum over unit vectors $u$  proves \eqref{20260724eq70}.
\end{proof}

\begin{prop}\label{20260724prop03}
Let ${\bf F}=(P,Q)$ be the $\C^2$-valued holomorphic function constructed in Proposition~\ref{20260724prop02}. Then
$$
\|{\bf F}'(z)\|_{\C^2}\lesssim \frac{1}{1-|z|},\qquad z\in\D.
$$
Consequently, $P,Q\in\calB$.
\end{prop}
\begin{proof}
For each $n\ge1$, let ${\bf H}_n$ be the $\C^2$-valued polynomial defined as in Proposition \ref{20260724prop02}. By Proposition~\ref{20260724prop01} and \eqref{20260724eq60},
$$
\deg{\bf H}_n\le2^{n-1}-1
\qquad\text{and}\qquad
\sup_{z\in\D}\|{\bf H}_n(z)\|_{\C^2}\le1.
$$
It follows from Lemma~\ref{20260724lem01} that
$$
\sup_{z\in\D}\|{\bf H}_n'(z)\|_{\C^2}
\le\textrm{e}\cdot2^{n-1}.
$$
Therefore, if $0<|z|=r<1$, then
$$
\|{\bf G}_n'(z)\|_{\C^2}
\le2^{n-1}r^{2^{n-1}-1}\|{\bf H}_n(z)\|_{\C^2}
+r^{2^{n-1}}\|{\bf H}_n'(z)\|_{\C^2}
\le(1+\textrm{e})2^{n-1}r^{2^{n-1}-1}.
$$
Moreover, 
$$
\sum_{n=1}^{\infty}2^{n-1}r^{2^{n-1}-1}
\le1+2\sum_{j=1}^{\infty}r^j
=\frac{1+r}{1-r}
\le\frac{2}{1-r}.
$$
For every $0<R<1$, the preceding estimates show that the series $\sum_{n=1}^{\infty}{\bf G}_n'$ converges uniformly on $\{|z|\le R\}$. Hence, 
$$
\|{\bf F}'(z)\|_{\C^2}
\le\sum_{n=1}^{\infty}\|{\bf G}_n'(z)\|_{\C^2}
\le\frac{2(1+\textrm{e})}{1-|z|},
\qquad 0<|z|<1.
$$
Finally, at $z=0$, we have ${\bf G}_1'(0)=(0,1)$ and ${\bf G}_n'(0)=0$ for $n\ge2$, so the same estimate holds after enlarging the absolute constant if necessary. The proof is complete. 
\end{proof}

\bigskip 

\section{Proof of Theorem \ref{Blochmainthm}} \label{Sec03}

In Proposition~\ref{20260724prop03}, we showed that the limiting procedure in Algorithm~\ref{20260724alg01} produces two Bloch functions $P$ and $Q$. In this section, we prove that they satisfy the properties asserted in Theorem~\ref{Blochmainthm}. We divide the proof into two steps.

\subsection{Step I: Proof of Theorem \ref{Blochmainthm} for points near $\T$}

We first need the following proposition.

\begin{prop}\label{20260724prop04}
There exists an absolute constant $A>0$ such that the following holds. Let $1/2<r<1$, let $\zeta\in\T$, and let $N\ge1$ be the unique integer satisfying
$$
2^{-N-1}<1-r\le2^{-N}.
$$
Then
$$
\|{\bf F}(r\zeta)-{\bf F}_N(\zeta)\|_{\C^2}\le A.
$$
In particular, $A$ is independent of $r$, $N$, and $\zeta$.
\end{prop}
\begin{proof}
By the definitions of ${\bf F}$ and ${\bf F}_N$,
\begin{equation} \label{20260725eq01}
{\bf F}(r\zeta)-{\bf F}_N(\zeta)
=\sum_{n=1}^N\bigl({\bf G}_n(r\zeta)-{\bf G}_n(\zeta)\bigr)
+\sum_{n=N+1}^{\infty}{\bf G}_n(r\zeta).
\end{equation} 
For $1\le n\le N$, Proposition~\ref{20260724prop01} gives
$$
\deg{\bf G}_n\le2^n-1
\qquad\text{and}\qquad
\sup_{\xi\in\T}\|{\bf G}_n(\xi)\|_{\C^2}=1,
$$
which, by Lemma~\ref{20260724lem01}, further implies that
$$
\sup_{z\in\overline{\D}}\|{\bf G}_n'(z)\|_{\C^2}\le\textrm{e}\cdot2^n.
$$
Integrating along the radial segment joining $r\zeta$ and $\zeta$, we obtain
$$
\|{\bf G}_n(r\zeta)-{\bf G}_n(\zeta)\|_{\C^2}
\le\textrm{e}(1-r)2^n.
$$
Therefore,
\begin{equation} \label{20260725eq02}
\sum_{n=1}^N\|{\bf G}_n(r\zeta)-{\bf G}_n(\zeta)\|_{\C^2}
\le2\textrm{e}(1-r)2^N
\le2\textrm{e}.
\end{equation} 
On the other hand, recall that 
$$
{\bf G}_n(z)=z^{2^{n-1}}{\bf H}_n(z),
\qquad
\sup_{z\in\D}\|{\bf H}_n(z)\|_{\C^2} \le 1.
$$
Therefore, using $r^s\le \textrm{e}^{-(1-r)s}$ for $0<r<1$ and $s \ge 0$, we have
\begin{align} \label{20260725eq03}
\sum_{n=N+1}^{\infty}\|{\bf G}_n(r\zeta)\|_{\C^2}
&\le\sum_{n=N+1}^{\infty} r^{2^{n-1}}  \le \sum_{n=N+1}^\infty \textrm{e}^{-(1-r)2^{n-1}} \nonumber \\
&= \sum_{n=N+1}^\infty e^{-(1-r)2^N \cdot 2^{n-1-N}} \le \sum_{k=0}^{\infty}\textrm{e}^{-2^{k-1}}<+\infty 
\end{align}

Finally, combining \eqref{20260725eq01}, \eqref{20260725eq02}, and \eqref{20260725eq03}, we deduce that 
\begin{equation} \label{20260725eq04}
\|{\bf F}(r\zeta)-{\bf F}_N(\zeta)\|_{\C^2}
\le2\textrm{e}+\sum_{k=0}^{\infty}\textrm{e}^{-2^{k-1}}.
\end{equation} 
The proof is complete.
\end{proof}

Our first result in this section shows that Theorem \ref{Blochmainthm} holds for points near $\T$. 

\begin{thm}\label{20260724thm01}
There exist $r_0\in(0,1)$ and a constant $c_0>0$ such that
$$
|P(z)|+|Q(z)| \simeq \|{\bf F}(z)\|_{\C^2}
\ge c_0\left(\log\frac{1}{1-|z|}\right)^{1/2}
$$
whenever $r_0\le|z|<1$.
\end{thm}

\begin{proof}
Let $z=r\zeta$, where $1/2<r<1$ and $\zeta\in\T$, and let $N\ge1$ be the unique integer satisfying
$$
2^{-N-1}<1-r\le2^{-N}.
$$
By Proposition~\ref{20260724prop01},
$$
\|{\bf F}_N(\zeta)\|_{\C^2}=\sqrt{N+1}.
$$
Hence, Proposition~\ref{20260724prop04} and the triangle inequality give
$$
\left|\|{\bf F}(r\zeta)\|_{\C^2}-\sqrt{N+1}\right|
\le\|{\bf F}(r\zeta)-{\bf F}_N(\zeta)\|_{\C^2}
\le A.
$$
Here, $A$ is the constant defined as in Proposition \ref{20260724prop04}; see \eqref{20260725eq04}. Moreover, the choice of $N$ gives
$$
N\log2\le\log\frac{1}{1-r}<(N+1)\log2.
$$
Choose $r_0\in(1/2,1)$ sufficiently close to $1$ so that $\sqrt{N+1}\ge2A$ whenever $r_0\le r<1$. Then
$$
\|{\bf F}(r\zeta)\|_{\C^2} \ge\frac{1}{2}\sqrt{N+1}.
$$
Since $\sqrt{N+1}$ is comparable to $\left(\log\frac{1}{1-r}\right)^{1/2}$, the desired estimate follows.
\end{proof}

Note that Theorem~\ref{20260724thm01} already provides a satisfactory answer to Open Question~\ref{20260724op01}.

\begin{cor} \label{20260725cor01}
    There exist three functions $f_1, f_2, f_2 \in \calB$ such that
    $$
    |f_1(z)|+|f_2(z)|+|f_3(z)| \gtrsim \left(\log \frac{1}{1-|z|} \right)^{1/2}, \qquad z \in \D. 
    $$
\end{cor}

\begin{proof}
The desired corollary follows by taking $f_1=P$, $f_2=Q$, and $f_3=1$, where $P$ and $Q$ are the Bloch functions constructed in Theorem~\ref{20260724thm01}.
\end{proof}

\subsection{Step II: Proof of Theorem \ref{Blochmainthm} for all points in $\D$}

In the second part of this section, we modify the construction used in Theorem~\ref{20260724thm01} to extend its conclusion to every point of $\D$, thereby completing the proof of Theorem~\ref{Blochmainthm}. We need the following elementary lemma. 

\begin{lem}\label{20260724lem02}
Let $0<\rho<1$ and $\D_\rho:=\{z\in\C:|z|\le\rho\}$. If ${\bf F}:\D\to\C^2$ is holomorphic, then, for every $\varepsilon>0$, there exists $a\in\C^2$ such that
$$
\|a\|_{\C^2}<\varepsilon
\qquad\text{and}\qquad
0\notin({\bf F}+a)(\D_\rho).
$$
\end{lem}
\begin{proof} 
Since ${\bf F}$ is continuously differentiable on a neighborhood of $\D_\rho$, it is Lipschitz on $\D_\rho$ as a map from $\R^2$ to $\R^4$. Since Lipschitz maps do not increase Hausdorff dimension,
$$
\dim_{\mathrm H}{\bf F}(\D_\rho)
\le\dim_{\mathrm H}\D_\rho
=2<4.
$$
Thus ${\bf F}(\D_\rho)$ has zero four-dimensional Lebesgue measure and hence empty interior in $\C^2$. Therefore, for every $\varepsilon>0$, we may choose
$$
a\in\{w\in\C^2:\|w\|_{\C^2}<\varepsilon\}\setminus\bigl(-{\bf F}(\D_\rho)\bigr).
$$
Then ${\bf F}(z)+a\ne0$ for every $z\in\D_\rho$.
\end{proof}

\begin{proof}[Proof of Theorem~\ref{Blochmainthm}]
Let ${\bf F}=(P,Q)$ be the $\C^2$-valued holomorphic function constructed in Theorem \ref{20260724thm01}. Therefore, we may choose $\rho\in(0,1)$ sufficiently close to $1$ so that
$$
\|{\bf F}(z)\|_{\C^2}\ge2,\qquad \rho\le|z|<1.
$$
Applying Lemma~\ref{20260724lem02} with $\varepsilon=1$, we obtain $a\in\C^2$ such that
$$
\|a\|_{\C^2}<1
\qquad\text{and}\qquad
0\notin({\bf F}+a)(\D_\rho).
$$
Set
$$
\widetilde{\bf F}:={\bf F}+a=:(\widetilde P,\widetilde Q).
$$
It is clear that $\widetilde P,\widetilde Q\in\calB$
Since adding constants does not change the Bloch seminorm.

If $\rho\le|z|<1$, then
$$
\|\widetilde{\bf F}(z)\|_{\C^2}
\ge\|{\bf F}(z)\|_{\C^2}-\|a\|_{\C^2}
\ge\frac12\|{\bf F}(z)\|_{\C^2} \gtrsim \left(\log \frac{1}{1-|z|} \right)^{1/2}.
$$
On the other hand, $\widetilde{\bf F}$ does not vanish on $\D_\rho$. Hence, by continuity and compactness,
$$
\min_{z\in \D_\rho}\|\widetilde{\bf F}(z)\|_{\C^2}>0.
$$
Hence 
$$
\|\widetilde{\bf F}(z)\|_{\C^2}
\gtrsim \left(\log\frac{\textrm{e}}{1-|z|}\right)^{1/2},
\qquad z\in\D_\rho.
$$
Combining the two estimates above and rescaling the functions by a suitable absolute constant completes the proof of Theorem~\ref{Blochmainthm}.
\end{proof}

\end{document}